\documentclass[12pt, a4paper]{amsart} 
\usepackage[T2A]{fontenc}
\usepackage[utf8]{inputenc}
\usepackage[english]{babel}
\usepackage{amsmath}
\usepackage{amssymb}
\usepackage{bbm}
\usepackage{comment}
\usepackage{enumerate}   
\usepackage{graphicx}
\usepackage{xcolor}
\definecolor{glamour}{rgb}{1, 0, 0.5}
\usepackage{tikz}
\usepackage{tikz-3dplot}
\usetikzlibrary{patterns}
\usepackage{geometry}

\newtheorem{theorem}{Theorem}
\newtheorem{proop}{Proposition}

\newtheorem{lemma}{Lemma}

\theoremstyle{remark}
\newtheorem*{rem}{Remark}

\def\XXint#1#2#3{{\setbox0=\hbox{$#1{#2#3}{\int}$ }
\vcenter{\hbox{$#2#3$ }}\kern-.6\wd0}}

\newcommand{\vertiii}[1]{{\left\vert\kern-0.25ex\left\vert\kern-0.25ex\left\vert #1 
    \right\vert\kern-0.25ex\right\vert\kern-0.25ex\right\vert}}

\begin{document}

\mbox{}
\nocite{*}
\title[
Zero sets of harmonic polynomials.
]{
On the zero sets of  harmonic  polynomials.
}
\author{Ioann Vasilyev}
\address{St.-Petersburg Department of V.A. Steklov Mathematical Institute, Russian Academy of Sciences (PDMI RAS), Fontanka 27, St.-Petersburg, 191023, Russia
}
\email{ivasilyev@pdmi.ras.ru, ioann.vasilyev@cyu.fr}

\begin{abstract}
In this paper we consider nonzero harmonic functions vanishing on some subsets of $\mathbb R^n$. We give a positive solution to  Problem 151 from the Scottish Book posed by R. Wavre in 1936. In more detail, we construct a nonzero harmonic polynomial that vanishes on the edges of the unit cube. Moreover, using harmonic morphisms we build  new nontrivial families of harmonic polynomials that vanish at the same set in the unit ball in $\mathbb R^n$ for all $n\geq 4$. This extends certain results by Logunov and Malinnikova from~\cite{logunov2015ratios}. We also present new results on  harmonic functions in the space whose zero sets are unions of affine codimension two subspaces.
\end{abstract}

\maketitle
\section{Introduction}

Let $d$ be a natural number and let $\Omega\subset \mathbb R^d$ be a domain. Throughout this paper we denote by $\mathcal H(\Omega)$  the space of harmonic functions in $\Omega$. Recall that a function is called harmonic in a domain if its Laplacian is equal to zero everywhere in this domain.
 We call a subset $S\subset \Omega$ a uniqueness set for the space $\mathcal H(\Omega)$ if 
$$f|_S=0, f\in \mathcal H(\Omega) \text{ \;\; implies \;\; } f(x)=0 \text{ \;\; for all } x\in \Omega.$$
In this paper we  focus only on  uniqueness sets for the space $\mathcal H(\mathbb R^d)$.

\begin{rem}
	All harmonic functions considered in this paper are harmonic in the whole space, unless we specify the contrary.
\end{rem}

If a harmonic function vanishes on a nonempty open set, then by analyticity it must vanish identically. On the other hand, the structure of zero sets of harmonic functions can be very complicated beyond this trivial case. 
Some examples of nontrivial sets of uniqueness for $\mathcal H(\mathbb R^d)$ are given in~\cite{kokurin2015sets}. These examples are given by non-analytic curves. A general description of uniqueness sets for the space $\mathcal H(\mathbb R^d)$ seems to be out of reach even if one restricts the considerations  to those sets that are finite unions of affine codimension two subspaces of $\mathbb R^d$.

In the planar case the corresponding question admits a simple and explicit solution. Indeed, let $E=\{p_j\}_{j=1}^N$ be a finite set with $p_j\in \mathbb C$. Then the polynomial defined for $z\in \mathbb C$ by the following formula
$$P_E(z):=\prod_{j=1}^N(z-p_j)$$
is holomorphic, nonzero and vanishes at the set $E$. This means that the polynomial $\Re(P_E)$ is hence the desired harmonic function.

Recall the classical Problem 151 from the Scottish Book, posed by Rolin Wavre, see~\cite{mauldin2015scottish}.

\bigskip

\begin{quote}
`November 6, 1936. Prize: A ``fondue'' in Geneva. Original manuscript in French.

\bigskip

Does there exist a harmonic function defined in a region which contains a cube
in its interior, which vanishes on all the edges of the cube ? One does not consider
$f=0$.'
\label{quote:one}
\end{quote}
\bigskip

In this paper we construct such a function. More precisely, we have the following theorem.

\begin{theorem}
	\label{labell}
For any natural number $d\geq 3$, there exists a nonzero harmonic polynomial in $\mathbb R^d$ that vanishes on the $d-2$ skeleton of the unit cube $Q_d$.
	\end{theorem}
	For $0\leq k \leq d$ we denote by $S_{k,d}$ the $k$-skeleton of the unit cube $Q_d:=[-1/2,1/2]^d$.
By the $k$-skeleton we mean the union of $k$-dimensional faces
of the unit cube, that is to say the set of all points that belong to the boundary of the cube and that have at least $d-k$ coordinates equal either to $-1/2$ or to $1/2$.

	We emphasize  that the description of the zero and the nodal sets of harmonic functions is a very interesting classical topic that remains popular to this day. Probably, the first questions in this domain date at least to the second half of XIX-th century from the classical work of Maxwell see~\cite{maxwell1873treatise}. This work is related to the Maxwell conjecture which still remains open, see~\cite{zolotov2023upper} for recent important progress on this conjecture.  For other recent advances in the study of the zero and nodal sets of harmonic functions and other solutions of partial differential equations we cite for example the papers~\cite{mangoubi2013gradient},~\cite{logunov2015ratios},~\cite{logunov2016ratios},~\cite{greilhuber2024quadratic},~\cite{gichev2009some},~\cite{de2010geometric} and~\cite{de2020weighted} to name a few.
	
\medskip

We are now in a position to show how to construct the desired harmonic polynomial in Theorem~\ref{labell}. Consider the  function defined for  $x=(x_1,\ldots,x_d)\in \mathbb R^d$ by the following  formula
$$
f_d(x_1,\ldots,x_d):=\prod_{1\leq i<j\leq d}(x_i^2-x_j^2).
$$
We remark that the function $f_d$ resembles the $d$-variate Vandermonde polynomial, see e.g.~\cite{macdonald1998symmetric}.
\begin{rem}
	Notice that the polynomial $f_d$ defined just above is not the unique nontrivial function that satisfies the assumptions of  Theorem~\ref{labell}. Another function that works is defined for $x\in \mathbb R^d$ by
	$$
g_d(x):=x_1\cdot\ldots \cdot x_d \cdot \prod_{1\leq i<j\leq d}(x_i^2-x_j^2).
$$
\end{rem}

We do not know a complete classification of all harmonic functions that vanish at the edges of the unit cube.

\begin{rem}
	Observe that the function $h$ defined for $x\in \mathbb R^d$ by
	$$h(x):=e^{x_1}\cdot \ldots \cdot e^{x_d} \prod_{1\leq i<j\leq d}(x_i-x_j)$$ 
	satisfies 
	$$
\Delta h(x)= d h(x),
$$
which means that the function  $h$ is a non-identically zero eigenfunction of the Laplacian on $\mathbb R^d$.
\end{rem}

Here is the main technical auxiliary result of the current paper.
\begin{lemma}
\label{lem1}
Let $f_d$ be as above and let $d\geq 2$ be a natural number. We have
$$\Delta f_d(x)=0,$$
for all $x\in \mathbb R^d$, where $\Delta=\sum^d_{k=1} \partial^2/\partial x_k^2$ denotes the Laplacian operator.
\end{lemma}

Indeed, the proof of Theorem~\ref{labell} boils down to that of Lemma~\ref{lem1}, since for each $d-2$ dimensional face $e$ of the skeleton $S_{d-2,d}$ of the unit cube of dimension $d\geq 3$ and  each point $x\in e$ one has either $x_k=x_l=\pm 1/2$ or $x_k=-x_l=\pm 1/2$ for some pair of coordinates $x_k, x_l$ with $(k,l)\in [1,\ldots, d]^2$. This  means that $x_k^2-x_l^2=0$ which in turn implies that  $f|_e=0$ and hence Theorem~\ref{labell} follows from Lemma~\ref{lem1}.


 We shall not give proofs of the statements contained in the two remarks just above Lemma~\ref{lem1}, since they are essentially the same as that of Theorem~\ref{labell}.

\medskip

Besides the single-valued harmonic polynomial our construction can be used to produce a family of multivalued harmonic functions in the sense of Wavre's `fonctions de passage', see~\cite{MR1556876} for definition. Because each edge of the cube skeleton determines an independent loop in the complement, one can assign arbitrary additive jumps (monodromy) to these loops and, by integrating the corresponding harmonic $1$-forms, obtain  multivalued harmonic functions defined on the compliment of the hypercube $d-2$ skeleton. This provides a highly symmetric model of the kind of multivalued harmonic maps considered in recent works by Donaldson~\cite{donaldson2025twistor} and Yan~\cite{yan2025construction}. The difference is that here one can prescribe independent monodromy parameters rather than a single constant monodromy.

\medskip

We now mention a question with connection to  Theorem~\ref{labell} which we leave open. Given a finite union $U$ of \textit{affine} codimension two subspaces in $\mathbb R^d$, under which condition on $U$ does there exist a nonzero harmonic function $F$ such that $F|_U=0$ ? Note that this question is consistent with the main results of the current paper since a  harmonic function that vanishes on an edge of a cube (which is a segment) should actually vanish on the whole line containing this edge. In effect, this fact holds in a more general setting of real analytic functions (harmonic functions are well known to be real analytic see~\cite{axler2013harmonic}). This means that a real analytic function that vanishes at a relatively open subset of an analytic subvariety, should vanish at the whole subvariety see e.g.~\cite{enciso2013some}. In the paper~\cite{flatto1966level} two particular cases of this question have been considered. Namely the case of linear one dimensional subspaces and that of parallel lines, both in the plane.  

\bigskip

Another goal of this paper is to construct   new families of harmonic polynomials that vanish at the same set in the unit ball in $\mathbb R^n$ for all $n\geq 4$. To do this, we use the theory of harmonic morphisms. 
\begin{theorem}
	\label{labell2}
In each space $\mathbb R^{2k}, k\geq 2$ there exist  irreducible nondegenerate quadratic homogeneous harmonic polynomials  that divide infinitely many linearly independent harmonic polynomials.

In each space $\mathbb R^{2k+1}, k\geq 2$ there exist irreducible nondegenerate  harmonic polynomials that divide infinitely many linearly independent harmonic polynomials.
\end{theorem}

By a nondegenerate polynomial we mean any polynomial $P$ whose  gradient does not lie in any proper coordinate subspace, i.e. for every nonzero vector $\omega\in \mathbb R^n$ there exists $x\in \mathbb R^n$ such that $\omega.\nabla P(x)\neq 0.$ We also prefer to explicitly say that by an irreducible polynomial we mean an irreducible polynomial over $\mathbb C$.

Observe that our Theorem~\ref{labell2} extends results by Logunov and Malinnikova from~\cite{logunov2015ratios}, Section 4. Indeed, in~\cite{logunov2015ratios} the authors constructed such a family of polynomials in dimension four using a different idea that axially-symmetric harmonic functions in  $\mathbb R^4$ can be reduced to ordinary harmonic functions of two variables, see~\cite{khavinson1991reflection}. The authors also construct in~\cite{logunov2015ratios} similar  examples of harmonic functions using closely related ideas in dimensions six and nine. However, in the language of Fuglede's harmonic morphism theory, see~\cite{fuglede1978harmonic}, the Logunov--Malinnikova constructions in Section 4 of their paper are examples of $h$-harmonic morphisms. These are maps that pull back every harmonic function on the target space to a function satisfying $\mathrm{div}(h\nabla u)=0$ on the source space. For instance, in $\mathbb{R}^9$ one can choose $h$ in the following manner $h(x)=\rho_1\rho_2\rho_3$, where  $\rho_j= \sqrt{x_j^2+x_{j+1}^2+x_{j+2}^2}$ for $j=1,4,7$. By contrast, the families of harmonic polynomials constructed in the second part of our paper arise from ordinary harmonic  morphisms from $\mathbb R^n$ to $\mathbb R^2$ (in Fuglede's sense, with $h\equiv 1$). Those thus pull back harmonic  functions on the target space $\mathbb R^2$ to harmonic functions on $\mathbb{R}^n$.

Here is another open question with connection to our Theorem~\ref{labell2}. This is the case where the ambient space is $\mathbb R^3$. This question was posed by M. Agranovsky~\cite{agranovsky2000problem}.

Let us stress that two parts of this paper are connected since they both handle the problem of constructing harmonic functions with prescribed vanishing sets in Euclidean spaces, combining classical explicit constructions with modern methods based on harmonic morphisms that are central in differential geometry.

\bigskip

We also obtain a geometric characterization of triangular prisms in $\mathbb{R}^3$ that support nontrivial global harmonic functions vanishing on the boundary. Here is the corresponding third main result of the current paper. 

\begin{theorem}\label{thm:prism}
Let 
\[
W=\{(0,y,z)\in\mathbb{R}^{3}\}
\]
be a vertical coordinate plane and let $P\subset W$ be a nondegenerate
triangle.  
For each side $S_j$ of $P$, let $\Pi_j\subset\mathbb{R}^{3}$ be the vertical
plane over $S_j$, i.e.
\[
\Pi_j=\{(x,y,z):\ (0,y,z)\in\operatorname{aff}(S_j)\},\qquad j=1,2,3.
\]
Let
\[
C_P:=\{(x,y,z):\ (0,y,z)\in P,\ x\in\mathbb{R}\}
\]
be the infinite prism over $P$, so that $\partial C_P=\Pi_1\cup\Pi_2\cup\Pi_3$.

Then the following are equivalent:
\begin{enumerate}
    \item There exists a nonzero harmonic function $f\in \mathcal H(\mathbb{R}^3)$ such that
    \[
    f|_{\partial C_P}=0.
    \]
    \item The triangle $P$ tiles the plane $W$ by reflections in its sides,
    i.e.\ the group generated by reflections in the three lines containing the
    sides of $P$ is discrete and its orbit of $P$ forms a tiling of $W$.
\end{enumerate}

Moreover, by the theorem of Bárány--Frankl--Maehara, the triangles with property
{\rm (2)} are, up to similarity, exactly the four with interior angles
\[
(60^\circ,60^\circ,60^\circ),\quad
(30^\circ,30^\circ,120^\circ),\quad
(30^\circ,60^\circ,90^\circ),\quad
(45^\circ,45^\circ,90^\circ).
\]
\end{theorem}

Consider a finite family $A=(A_j)_{1\leq j\leq m}$ of affine hyperplanes in $\mathbb R^d$. The behavior of global harmonic functions that vanish on the union 
$\cup_j A_j$ is controlled by three distinct analytic mechanisms, which interact in subtle ways.

\begin{itemize}

\item Geometric constraint.
If $A$ has a polyhedral chamber (a bounded convex polytope cut out by the hyperplanes), then one may attempt to construct harmonic functions by solving Dirichlet problems in that chamber. In this regime, the boundary geometry of the chamber plays an important role.

\item Iterated reflections.
Starting with a harmonic function in a single chamber, one can extend it across any hyperplane  on which it vanishes by  the Schwarz reflection principle.
If the resulting reflected hyperplanes accumulate in many directions (e.g. if successive reflections generate arbitrarily small angles) analyticity forces the extended function to vanish identically. This mechanism explains the `rigidity' or uniqueness phenomena that occur in many non-symmetric arrangements.

\item Symmetry.
In contrast, for arrangements with a high degree of symmetry (most notably those arising from Coxeter groups), the iterated Schwarz reflections generate a discrete group, and one obtains nontrivial global harmonic functions vanishing on all hyperplanes in the arrangement. Classical examples include the coordinate hyperplanes, dihedral arrangements in the plane, and the prisms over reflection-tiling triangles appearing in Theorem~\ref{thm:prism}.
\end{itemize}

These three mechanisms  form the analytic framework behind the uniqueness/ nonuniqueness behavior for hyperplane arrangements. A full classification of arrangements for which non-identically zero global harmonic functions vanish on $A$ seems to require a deeper analysis of these reflection dynamics and lies beyond the scope of this paper.


\subsection*{Acknowledgments}
The author is deeply grateful to Sergey Kislyakov and to Mikhail Vasilyev for a number of helpful discussions.

\section{Proof of Lemma~\ref{lem1}}

So, we concentrate on the proof of Lemma~\ref{lem1}.

\begin{proof} 
 
There is a way to derive this result from~\cite{konig2001eigenvalues}, Lemma 3.1, but we prefer to give a self-contained new alternative proof here. To this end, let $V$ denote the Vandermonde polynomial defined by
$$
V(y):=\prod_{1\leq i<j\leq d}(y_i-y_j).
$$ 
 for all $y\in \mathbb R^d$.

We know that for all $y\in \mathbb R^d$ it holds that
\begin{multline*}
\Delta f_d(y)=\sum_{i=1}^d \frac{\partial^2}{\partial y_i^2} V(y_1^2,\ldots, y_d^2) \\
= 4\sum_{i=1}^d y_i^2\frac{\partial^2}{\partial y_i^2}V(y_1^2,\ldots, y_d^2)+2\sum_{i=1}^d \frac{\partial}{\partial y_i}V(y_1^2,\ldots, y_d^2).
\end{multline*}

Consider an auxiliary differential operator
$$L:=\sum_{j=1}^d L_j,$$
where $L_j:=2x_j\partial^2_j + \partial_j$. Notice that we now only need to show that $LV=0$ identically. By a well-known property of the Vandermonde polynomial,
$$V(x_1,\ldots,x_d)=\det[C_1,\ldots,C_d],$$
where $C_m:=(x_1^{m-1},\ldots, x_d^{m-1})^t$ are the columns of the Vandermonde matrix. It follows that
\begin{equation}
\label{lalalalal}
LV=\sum_{j=1}^d \det[C_1,\ldots, C_{j-1}, LC_j, C_{j+1},\ldots, C_d].
\end{equation}
We are now done since 
$$(2x_j \partial^2_j +\partial_j)x_j^{m-1}=(m-1)(2m-3)x_j^{m-2},$$
which means that all determinants in~\eqref{lalalalal} have either two proportional columns or a zero column and hence are equal to zero.
\end{proof}

\section{A general mechanism producing infinite families of harmonic polynomials vanishing on the same set in the unit ball}

Let us now prove Theorem~\ref{labell2}.

\begin{proof}
Let $n$ be a natural number and consider for instance the mapping $\phi=(\phi_1,\phi_2): \mathbb R^{2n}\rightarrow  \mathbb R^2$ defined by 
$$\phi_1(x):=x^2_1-x^2_2+x^2_3-x^2_4+\ldots+x_{2n-1}^2-x_{2n}^2$$ 
and 
$$ \phi_2(x):=2x_1x_2+2x_3x_4+\ldots +2x_{2n-1}x_{2n}.$$ Then, according to~\cite{ou1997quadratic}, Theorem 5.11  this mapping is a harmonic morphism. Thus the irreducible quadratic homogeneous harmonic polynomial $$\phi_1(x)=x^2_1-x^2_2+x^2_3-x^2_4+\ldots+x_{2n-1}^2-x_{2n}^2$$  that depends nontrivially on all $2n$ variables of  $\mathbb R^{2n}$  divides infinitely many linearly independent harmonic polynomials, for example $$P_k(x):=\Re(\phi_1(x)+i\phi_2(x))^{2k+1}$$ with natural $k$ will work. Let us formally prove that polynomials $P_k$ form a nontrivial family. Notice that in $\mathbb R^2$ the polynomial $\Re(u+iv)^{2k+1}$ has the following form $$\Re(u+iv)^{2k+1}=u^{2k+1} + uvT_k(u,v),$$  
where $T_k$ is a polynomial. This means that the polynomial $P_k$  has unimodular coefficients in front of the monomials  
$$x_1^{4k+2},\ldots, x_{2n}^{4k+2}$$
 and is thus nontrivial. Moreover, it holds that $\deg(P_k)=4k+2$ so the family is linearly independent.

\bigskip

Let us now turn to the case of odd dimensions $m:=2k+1$ with $k\geq 2$. For odd dimensions the standard harmonic morphisms are slightly more involved, but a convenient explicit construction is available in the literature (Baird--Wood,~\cite{baird2003harmonic}). For the reader's convenience we outline the necessary parts of this construction.
\begin{enumerate}[(i)]
\item  $\phi_1$ is a homogeneous polynomial of degree $m-3$;
\item   $\phi_1$ is irreducible and its gradient is nowhere contained in a proper coordinate subspace;
\item  $\phi$ is surjective and horizontally weakly conformal.
\end{enumerate}
For completeness we indicate the verification of (ii), although no new ideas are required.

We use the result given by Example 5.2.7 in~\cite{baird2003harmonic}. That result says that the polynomial mapping
$$(\varphi_1(x),\varphi_2(x)):=\xi_1(z)x_1+\ldots+\xi_{m-2}(z)x_{m-2},$$
where $x=(x_1,\ldots,x_m)\in \mathbb R^m$ and $z:=x_{m-1}+ix_m$ is a harmonic morphism from $\mathbb R^m$ to $\mathbb C$ once we know that
$$\sum_{j=1}^{m-2}\xi^2_j(z)=0$$
identically. In order to satisfy this last condition, we use ansatz 6.8.6 on the page 200 in~\cite{baird2003harmonic}, which gives
$$\xi:=(1-g^2,i(1+g^2),-2g),$$
where we pose $g^2:=g_1^2 + \ldots + g_{m-4}^2$ and $g:=(g_1,\ldots, g_{m-4})$ for any choice of polynomials $(g_1,\ldots, g_{m-4}).$ Thus it suffices to replace the mapping $(\phi_1,\phi_2)$ from the previous considerations by the mapping $(\varphi_1,\varphi_2)$.

The only two things that need to be clarified in this case are the nondegeneracy of $\varphi_1$ and its irreducibility. To this end we perform the following choice $$g_0(z):=(z, z^2, \ldots, z^{m-4}).$$ This choice gives rise to the following harmonic divisor
$$\varphi_{1,0}(x):=\sum_{j=1}^{m-2}x_jf_j(x_{m-1},x_m),$$
where $f_j$ are  polynomials. From here it is now easy to conclude that $\varphi_{1,0}$ is irreducible. For if not then $\varphi_{1,0}(x)=P(x_1,\ldots,x_m,z)Q(x_1,\ldots,x_m,z)$, where $P$ and $Q$ are certain polynomials. Hence, $P(0,z)Q(0,z)=0$ identically. From here we see that either $P(0,z)=0$ or $Q(0,z)=0$ identically. Observe that both can't be zero identically, since $\varphi_{1,0}(x)$ has degree $1$ in $x_j$. Without loss of generality, we suppose that $P(0,z)\equiv 0$ and $Q(0,z)\not\equiv0$. Since the polynomial $P$ is at most linear in $x_1,\ldots,x_{m-2}$ we conclude that there exist polynomials $A, \beta_0,\beta_1,\ldots, \beta_{m-2}$ all depending only on $x_{m-1}$ and $x_m$ such that
$$\varphi_{1,0}(x)=A(x_{m-1},x_m)\left(\beta_0(x_{m-1},x_m)+\sum_{j=1}^{m-2}x_j \beta_j(x_{m-1},x_m)\right).$$
Comparing coefficients, we get $A\beta_j=f_j$ for all $1\leq j\leq m-2$ and $A\beta_0=0$. Since $A$ is nonzero, we have $\beta_0=0$. Next, we see that $A$ is a common divisor of polynomials $f_j$ for all $1\leq j\leq m-2$. But, according to the construction, the polynomials  $f_j$ are coprime which follows, for instance from the facts that $f_3=2x_{m-1}$ and $f_4=x_{m-1}^2-x_{m}^2$ are coprime so any factor of $\varphi_{1,0}$ would divide both, which is impossible. Hence the factorization is trivial.

Regarding nondegeneracy, we need to show that $\omega.\nabla\varphi_{1,0}(x)=0$  for all $x\in \mathbb R^m$ implies that $\omega=0$. Writing $\omega=(a_1,\ldots,a_{m-2},b,c)$ and computing the gradient we obtain
$$\omega.\nabla\varphi_{1,0}(x)=\sum_{j=1}^{m-2}a_jf_j(z)+b\sum_{j=1}^{m-2}x_j\frac{\partial f_j}{\partial x_{m-1}}(z)+c\sum_{j=1}^{m-2}x_j \frac{\partial f_j}{\partial x_{m}}(z)=0.$$
Considering this scalar product as a polynomial of $x_1,\ldots,x_{m-2}$ we see that all its coefficients must vanish identically. This yields
$$\sum_{j=1}^{m-2}a_jf_j(z)=0$$
identically, forcing $a_j=0$ for all $1\leq j\leq m-2$ by linear independence of $f_j$. It is left to show that $b=c=0$. We remark that we must have
$$b\partial_{m-1}f_j(z)+c\partial_{m}f_j(z)=0$$
identically for all $1\leq j\leq m-2$. But this is impossible if $(b,c)\neq (0,0)$ for our explicit choice of ansatz. Indeed, checking the last line for $f_3$ gives $b=0$ and for $f_4$ gives $c=0$.

\end{proof}

Notice that the last argument works for all $m\geq 3$. However, for $m=3$ this gives no nontrivial example that we wish to construct, and for $m=4$ we have already constructed one, so we restrict our considerations above to the case where $m\geq 5$. Note also that our result extends those from the fourth section of paper~\cite{logunov2015ratios}.

There are many other harmonic morphisms between Euclidean spaces. We specifically name those based on Hopf polynomials and refer the interested reader to the book~\cite{baird2003harmonic} for this and other related constructions.

Here is one more result that provides a general way of constructing  examples of linearly independent harmonic functions vanishing at the same set in a ball via `lifting ' lower dimensional examples.

\begin{proop}
Let $\{f_j\}_{j=1}^\infty$ be a linear independent family of functions satisfying $f_j\in \mathcal H(\mathbb R^n)$ with $n\geq 3$ for each $j$ and such that there exists $\delta>0$ satisfying $Z_{f_j}\cap B(0,\delta)=Z_{f_i}\cap B(0,\delta)$ for all $(i,j)\in \mathbb N^2$, and let $\varphi: \mathbb R^m \rightarrow \mathbb R^n$ be any nonconstant harmonic morphism. Then we have
\begin{enumerate}
\item \label{one} $f_j\circ \varphi \in \mathcal H(\mathbb R^m)$ for all $j\in \mathbb N$.
\item \label{two}  $\{f_j\circ \varphi \}_j$ is a linearly independent family.
\item \label{three}  There is $\delta_1>0$ such that $Z_{f_j\circ \varphi}\cap B(0,\delta)=Z_{f_i\circ \varphi}\cap B(0,\delta)$.
\end{enumerate}
\end{proop}
\begin{proof}
The first and the third properties above are evident. According to~\cite{ababou1999polynomes}, if $\varphi: \mathbb R^m \rightarrow \mathbb R^n$ is a global nonconstant harmonic morphism then it is surjective. Thus the second property follows as well. 
\end{proof}

\section{Harmonic Dirichlet data on triangular prisms}

The constructions in the first two sections rely on algebraic symmetries of the zero set, either via explicit polynomial factors or via the pullback of harmonic morphisms. It is natural to ask whether analogous phenomena exist in a purely geometric setting, where the vanishing set is not algebraic but instead arises as the boundary of an unbounded polyhedral region.
In this section we show that the same principle persists: the existence of a non-identically zero harmonic function vanishing on the boundary of a triangular prism is equivalent to a discretely generated reflection symmetry of the prism cross-section. This provides a geometric counterpart to the algebraic constructions in the earlier sections.

In this section we record a geometric criterion for the existence of nontrivial
harmonic functions on $\mathbb{R}^{3}$ vanishing on the boundary of an
unbounded triangular prism.  
The result is a three-dimensional reflection analogue of the classical
classification of planar reflection-tiling triangles due to
Bárány--Frankl--Maehara.

\begin{proof}
We first prove {\rm (1)} $\Rightarrow$ {\rm (2)}.
Assume there exists a nonzero harmonic function
$f\in \mathcal H(\mathbb{R}^{3})$ such that $f|_{\partial C_P}=0$. For each side $S_j$ of $P$, let $L_j\subset W$ be the line through $S_j$, and
let $\sigma_j$ denote the reflection in the vertical plane $\Pi_j$.
Since $f$ is harmonic near $\Pi_j$ and vanishes on $\Pi_j$, the Schwarz
reflection principle for harmonic functions implies
\begin{equation}\label{eq:odd}
    f(\sigma_j x)=-f(x), \qquad x\in\mathbb{R}^3.
\end{equation}

Let $G$ be the group generated by $\sigma_1,\sigma_2,\sigma_3$. 
Projecting to $W$, the induced action is exactly the group generated by
reflections in the lines $L_1,L_2,L_3$. By \eqref{eq:odd}, for every $g\in G$ we have
\[
f\equiv 0 \quad\text{on } g(\partial C_P).
\]

Suppose now that the triangular reflection group on $W$ is not discrete.
By a theorem of Bárány--Frankl--Maehara, the orbit of any vertex of $P$ under
reflections in $L_1,L_2,L_3$ is dense in $W$.  
Hence the union of the reflected lines $g(L_j)$ is dense in $W$, and the union
\[
\bigcup_{g\in G} g(\partial C_P)
\]
is dense in $\mathbb{R}^3$.
Since $f$ vanishes on this union and is real-analytic, it must vanish
identically, contradicting our assumption.
Thus the reflection group must be discrete, proving {\rm (2)}.

\smallskip

We now prove {\rm (2)} $\Rightarrow$ {\rm (1)}.
Assume that $P$ tiles $W$ by reflections in its sides.
Let $G$ be the generated reflection group acting on $\mathbb{R}^3$ via the
reflections $\sigma_1,\sigma_2,\sigma_3$.
Because the cross-section tiles $W$, the prisms $g(C_P)$, $g\in G$, tile
$\mathbb{R}^3$.

Choose a nonzero harmonic function $u$ on $C_P$ vanishing on the three vertical
faces. The existence of such a function can be justified by considering the product of the first Dirichlet eigenfunction of the triangle and an exponential function. Let $\lambda$ be the first Dirichlet eigenvalue and let $\phi$ be the first Dirichlet eigenfunction of the triangle $P$.
Then  the function $u(x,y,z)=\phi(y,z)e^{-\sqrt{\lambda} x}$ is harmonic in $C_P$ and vanishes on $\partial C_P$.  
Extend $u$ to all of $\mathbb{R}^{3}$ by successive Schwarz reflections.
For each generator $\sigma_j$ we set
\[
U(\sigma_j x) := -u(x).
\]
To see that this  is well defined, recall that the reflections in the
three sides of $W$ generate a planar Coxeter reflection group. In particular,
the chambers $g(C_P)$, $g\in G$, have pairwise disjoint interiors, and every
point of $\mathbb{R}^3$ lies in exactly one such chamber. Therefore, if
$g_1 x = g_2 x$ with $x \in C_P$, then necessarily $g_1 = g_2$, and the value
$U(gx)$ is uniquely determined. In addition, since all defining relations in a
reflection group have even length, the homomorphism
$\chi : G \to \{\pm1\}$ given by $\chi(\sigma_j)=-1$ is well defined, ensuring
that the signs used in the Schwarz reflections are consistent.

Call the resulting function $f$.
It is harmonic, real-analytic, antisymmetric under each $\sigma_j$, and hence
vanishes on each $\Pi_j$.
In particular, $f|_{\partial C_P}=0$.
Since $u$ is nonzero on $C_P$, the extension $f$ is nonzero on
$\mathbb{R}^{3}$. This establishes {\rm (1)}.
\end{proof}

Thanks to the work by Logunov and Malinnikova, see~\cite{logunov2015ratios} in the second statement of the result just above, in the case where a triangle is replaced by the unit square the corresponding function can be written explicitly. Indeed, it can be verified that the function
$$\psi(x):=\sin(a_1x_1)\cdot \ldots \cdot \sin(a_{d-1}x_{d-1})\sinh(a_{d}x_{d})$$
defined for all $x\in \mathbb R^d$ is harmonic, whenever real numbers $a_1,\ldots, a_d$ are such that holds $a_d^2=a_1^2+\ldots+a_{d-1}^2$. It is also clear that the function $\psi$ vanishes on the boundary of the corresponding prism.

It would be interesting to determine whether  such explicit formulas exist in the case of other $P$ and in the case of several dimensions. Another open question is to classify \textit{all} nonzero harmonic functions that vanish on the boundary of a half-strip or a half-prism.

\bibliographystyle{plain}
\bibliography{151}
\end{document}